\pdfoutput=1

\documentclass[a4paper,10pt,twocolumn]{article}
\usepackage[T1]{fontenc}
\usepackage{lmodern}
\usepackage{lipsum}
\usepackage{caption}
\usepackage{authblk}
\usepackage[top=2cm, bottom=2cm, left=2cm, right=2cm]{geometry}
\usepackage{fancyhdr}
\usepackage{url}
\usepackage{graphicx}
\usepackage{amssymb, gensymb}
\usepackage{amsmath}
\usepackage[authormarkuptext=id]{changes}
\usepackage{soul}
\usepackage{tikz}
\usetikzlibrary{patterns}
\usepackage{pgfplotstable, booktabs}  
\usepackage{lineno}
\usepackage{array}
\usepackage{amsthm}
\definecolor{mycolor1}{RGB}{130,220,202}
\definecolor{mycolor2}{RGB}{79,122,142}  
\definecolor{mycolor3}{RGB}{170,35,3}
\definecolor{mycolor4}{RGB}{207,170,114}
\definecolor{mycolor5}{RGB}{80,135,63}
\definecolor{mycolor6}{RGB}{255,140,190}

\begin{document}
 \title{\sf Second order formulation of boundary value problems in gradient elasticity.} 
 \author[1]{Antonios Charalambopoulos\thanks{acharala@math.ntua.gr} } 
 \author[1]{Evanthia Douka}
 \author[2]{Stelios Mavratzas} 
 
\affil[1]{School of Applied Mathematical and Physical Sciences, National Technical University of Athens, GR 15780 Athens, Greece}
\affil[2]{Department of Informatics and Computer Technology, Technological Educational Institution of Western Macedonia, Kastoria 52100, Greece}

\maketitle
\thispagestyle{empty}
\abstract{
A new formulation of boundary value problems in gradient elasticity is presented in this work. The main outcome is the construction of partial differential systems of second order, which are typically equivalent with the well known fourth order equation of gradient elasticity. Two alternative methodologies are developed and presented in the present work. The first approach is based on the framework of the pseudo-differential calculus and exploits the special characteristics of this approach to deal with the non local behavior of gradient elasticity. The second implementation is purely differential and is based on the augmentation of the independent variables of the problem. Under this concept, the constitutive equations of gradient elasticity become part of the differential system itself and the whole framework is reminiscent of the transformation of the wave equation to a first order differential system. Crucial issues like existence, uniqueness and stability of the corresponding initial boundary value problems are also encountered via the new concepts.}

%
\date{}
\section{Introduction}
\label{sec:intro}
Classical theory of linear elasticity is proved inadequate to describe size
and microstructural effects or to represent fields characterized by very high gradients of strains. 
This is merely assigned to the absence of internal parameters in the settlement of classical elasticity theory. Granular
materials, polymers, liquid crystals, porous media, solids with micro cracks, dislocations
and disclinations, composites are characteristic cases in which the classical elasticity theory fails in providing satisfactory modeling. 
This drawback of the classical theory has been confronted by other enhanced elastic theories where internal length-scale
parameters correlating the microstructure with the macrostructure are involved in the constitutive
equations of the considered elastic continuum. 
We mention here indicatively some primitive scientific boosts in literature as Cosserat elasticity theory \cite{Cosserat:1909},  
couple stresses theory \cite{Mindlin:1962}, \cite{Koitier:1964}, \cite{Toupin:1964}, multipolar theory of continuum mechanics
\cite{Green:1964}, higher-order strain-gradient elastic theory \cite{Mindlin:1964}, \cite{Mindlin:1965}, micromorphic, microstretch and micropolar
elastic theories \cite{Eringen:1999} and non-local elasticity \cite{Eringen:1992}, \cite{maugin:2006}.

In the majority of the works evoking the higher-order strain-gradient elastic theory, the stored dynamic energy of the elastic medium is not only produced by the mechanical work appeared when strains coexist with stresses but also additional storage of energy is produced via the interference of the gradient of the strains with the underlying double stresses \cite{Mindlin:1964}. In time dependent processes the inertia term is also influenced by the microstructure \cite{Mindlin:1964} and this influence leads to mixed type space-time derivatives in the differential law of the subsequent processes. For time harmonic problems and especially in scattering processes, the emergence of additional waves with dissipative characteristics - reflecting the internal multiple scattering phenomena - is indicative for the existence of the inner structure while the wave propagation of these waves inside waveguides with application interest \cite{acharala:2008}, \cite{acharala:2012} is illustrative for the coexistence of the multi scaling characteristic dimensions of the investigated media. The static and dynamic models describing the behavior of gradient elastic materials focus mainly on the determination of the displacement elastic field and this formulation is accompanied with boundary (initial) value problems of fourth order. This necessitates the implication of partial differential equations of higher order and several attempts have been performed to establish all the necessary analytical and numerical structural elements that build the appropriate framework for this effort \cite{acharala:2010}, \cite{acharalaCMES:2010}, \cite{Tang:2003}. 

In the present work we are interested in investigating the possibility of modeling the problems of gradient elasticity by formulations based on second order boundary value problems. This effort is justified by the willingness to create a framework in which all the well established traditional accomplishments, concerning second order equations, could be appropriately exploited to the gradient elasticity problems. As an example, the finite element calculus for higher order differential operators is demanding, time consuming and restrictive pertaining to the smoothness of the test functions across the interfaces of the elements, while the relevant realm referring to second order schemes is flexible, well studied and adaptable easily to a wide area of applications. Evidently, replacing a fourth order system by some new model of second order is not expected to be, at least by principle, a "bijective" transformation. In rough terms, we expect the original problem to be equivalent to a two stage scheme of two consecutive interrelated second order problems or alternatively to be equivalent with a second order scheme with more independent variables. This intuitive concept has been implemented in the present work giving birth to two separate alternative formulations of static boundary value problems in gradient elasticity. More precisely in Section \ref{sec:pereliminaries} we give the preliminary concepts concerning boundary value problems of gradient elasticity for the general case admitting also time dynamic phenomena. The first approach between the aforementioned constructed methods is presented in Section \ref{pseudo}. This method implies the involvement of the Poisson's and Green's operators and divides the original problem into two subsequent second order problems. However if the boundary conditions of the original problem, as these emerge on the basis of the relevant calculus of variations, are respected, then the boundary problem of gradient elasticity is transformed to a second order problem whose boundary condition is expressed via a pseudo-differential operator of order zero. This operator is the superposition of the classical trace operator and a pseudo-differential operator of order zero, which constitutes a non local operator reminiscent of the operators met in non local theories \cite{maugin:2006} with the essential difference that non locality is present in the boundary condition instead of the differential law. The developed methodology permits regularly the settlement of stability and convergence arguments characterizing the transmission of the original boundary value problem as the gradient parameters fade away and the gradient elasticity regime gives place to classical elasticity framework. The second approach is totally different and is based on the augmentation of the independent variables of the problem. The whole approach is presented in Section \ref{alter} and is based on the idea to reconsider the gradient of the strains as independent variables while their connection with the displacement field is incorporated in the partial differential equations of the problem substituting then the relevant constitutive equations. Between all the possible second order realizations of the problem only the one leading to a bounded below sesquilinear form is selected, in order to adapt the theory in the well established second order weak formulation. Existence and uniqueness of the solution in accordance with identification with the solution of the original problem are well established.

\section{The fundamental notions of gradient elasticity}
\label{sec:pereliminaries}

Let us consider a three-dimensional linear, gradient elastic body, in which we pay attention on a material volume $V$ confined by
a surface $S$, which is geometrically characterized by its normal vector 
$\widehat{{n}}$, which for simplicity is taken to be a continuous
vector field, fact reflecting the smoothness of the boundary $S$. The strain field $\epsilon_{ij},\ \ i,j=1,2,3$ does not contain enough information for the behavior of the material under stimulation and must be complemented with the gradient of the micro-deformation $\kappa_{ijk}$ and the relative deformation $\gamma_{ij}$ in order to form a set of adequate independent variables of the problem \cite{Mindlin:1964}. 

The potential energy per unit-macrovolume can be considered as a function of these variables
\begin{equation}
W=W(\epsilon_{ij},\kappa_{ijk},\gamma_{ij}).
\end{equation}%
The Mindlin form II assumes vanishing of the relative deformation and invokes so the gradient of the symmetric strain  ${\hat{\kappa}}_{ijk}={\partial}_{i} \epsilon_{jk}$ leading to a simpler representation of $W$:
\begin{equation}
W=W(\epsilon_{ij},{{\hat{\kappa}}_{ijk}})
\end{equation}%
Introducing the vector displacement field $u$, the strain field can be expressed in dyadic form as the strain elastic tensor
\begin{equation} 
\widetilde{e}=\frac{1}{2}\left( \nabla u+\left( \nabla u\right)
^{21}\right).
\end{equation}
The potential energy confined in the region $V$ is expressed as the integral 
\begin{equation}
U_{V}=\int_{V}\left[ \widetilde{{\tau} }\,: \, \widetilde{{e}}\mathbf{+}{{\left( 
\widetilde{{\mu} }\right) ^{321}\, \vdots \,\nabla \widetilde{{e}}}}\right] d\mathbf{r,} 
\end{equation}
where we recognize the contraction of the independent variables with the dual tensors. More precisely we encounter the Cauchy stress tensor 
\begin{eqnarray}
&&\widetilde{{\tau} } \left. \mathbf{=}
\right. 2\mu \widetilde{e}+\lambda 
\widetilde{{I}}\left( \nabla \cdot u\right)  \notag \\
&&\left. =\right. \mu 
\left( \nabla u+\left( \nabla u%
\right) ^{21}\right) +\lambda \widetilde{I}\left( \nabla \cdot u%
\right) 
\end{eqnarray}
and the double stress tensor \cite{Mindlin:1964} ${\widetilde{{\mu}}=\mu _{ijk}\widehat{{x}}_{i}
\widehat{{x}}_{j} \widehat{{x}}_{k}}$. These tensors share symmetry properties as follows 
\begin{equation}
\widetilde{{\tau} }=\widetilde{{\tau} }^{21},\;\widetilde{{\mu}}\mathbf{=}%
\widetilde{{\mu} }^{132}
\end{equation}%
The double stress tensor $\widetilde{{\mu}}$ is given by 
\begin{eqnarray}
\widetilde{{\mu} } &=&\frac{1}{2} {a_1}[
(\widetilde{I}\Delta u)^{312}+
\widetilde{{I}}\nabla \nabla \cdot u+
(\widetilde{{I}}\nabla \nabla \cdot u)^{312}
\notag \\
&+&
(\widetilde{{I}}\nabla \nabla \cdot u)^{132}]+
2{a_2}(\widetilde{{I}}\nabla \nabla \cdot u)^{312}
\notag \\
&+&
\frac{1}{2}{a_3}[
\widetilde{{I}}\Delta u+
\widetilde{{I}}\nabla \nabla \cdot u +
(\widetilde{{I}}\Delta u)^{132}+
(\widetilde{{I}}\nabla \nabla \cdot u)^{132}]
\notag \\
&+&
{a_4}[\nabla \nabla u+(\nabla \nabla u)^{231}]
\notag \\
&+&
\frac{1}{2}{a_5}
[
2(\nabla \nabla u)^{312}+
\nabla \nabla u+
(\nabla \nabla u)^{231}
] \label{mujune}
\end{eqnarray}
where ${a_i}, i=1,..5$ are the gradient constitutive parameters. 
We are in position to write Eq.(\ref{mujune}) in the condensed form
\begin{eqnarray} 
\widetilde{{\mu}} 
= {\mathcal{H}} \ \vdots \ \nabla \widetilde{e} \label{represh}
\end{eqnarray}
on the basis of a well defined polyadic ${\mathcal{H}}$ of sixth order, satisfying the symmetry relations
\begin{eqnarray} 
{\mathcal{H}}={{{\mathcal{H}}}}^{132456}={{{\mathcal{H}}}}^{123546}={{{\mathcal{H}}}}^{654321} 
\end{eqnarray}  
and whose specific form is omitted here for the sake of brevity. Apart from symmetry the tensor ${\mathcal{H}}$ merits positiveness in the sense that  
\begin{eqnarray} 
{(\nabla \widetilde{e})}^{321} \ \ \vdots {\mathcal{H}} \ \ \vdots \ \nabla \widetilde{e} \geq c a {|\nabla \widetilde{e}|}^{2} = c a {(\nabla \widetilde{e})}^{321} \ \vdots \ \nabla \widetilde{e} \label{ante}
\end{eqnarray}  
where $c$ is a constant independent of the gradient parameters and $a$ is the minimum of the non zero parameters $a_{i},\ \ i=1,2,...5$, provided that $a_{4}, a_{5}$ do not vanish simultaneously.  

Applying Hamilton's principle, we obtain
\begin{equation}
\delta \int_{t_{0}}^{t_{1}}\left( K_{V}-U_{V}\right) dt+\int_{t_{0}}^{t_{1}} \delta
W_{V} dt=0
\end{equation}
where $K_{V}$ stands for the kinetic energy in $V$ , while $ W_{V}$ is the work produced by external forces applied on the region $V$.  

Applying calculus of variation, we are in position to determine the differential law and the boundary conditions of our problem. First we state the induced PDE:
\begin{eqnarray}
\Delta^{\ast} \cdot u -\nabla \nabla\ : {\widetilde{{\mu}}}&=& -f+\rho \frac{\partial ^{2}}{\partial t^{2}}\left[ u-{h_{1}^{2}\nabla \nabla}
\cdot u \right. \nonumber \\ && \left. +{h_{2}^{2}\nabla \times \nabla} \times u\right], \label{basceq}
\end{eqnarray}
which is the well known equation of gradient elasticity. Using Eq.(\ref{mujune}) we express everything on the basis of the displacement field and acquire the fourth order differential equation of gradient elasticity 
\begin{eqnarray}
&&\left( \lambda +2\mu \right) \left( 1-{\xi _{1}^{2}\Delta} \right) \nabla
\nabla \cdot u \nonumber \\
&&-\mu \left( 1-{\xi _{2}^{2}\Delta} \right) \nabla
\times \nabla \times u\left. =\right.   \notag \\
&& -f+\rho \frac{\partial ^{2}}{\partial t^{2}}\left[ u-{h_{1}^{2}\nabla \nabla}
\cdot u+{h_{2}^{2}\nabla \times \nabla} \times u\right]
\end{eqnarray}
where the four parameters $\xi _{1}$, $\xi _{2}$, $h_{1}$ and $h_{2}$ are connected with the aforementioned constitutive parameters $a_{i}$ and reflect the influence of the microstructure. The parameters $\lambda, \mu$ are the Lam$\acute{\text{e}}$'s constants (slightly perturbed due to the microstructure). 
In the absence of microstructure, ${\xi}_{1}={\xi}_{2}=h_{1}=h_{2}=0$ and we recover the classical elasticity differential law.
The second product of the calculus of variations are the boundary conditions, which are divided in two sets:

a) The classical conditions: \newline 
i) Continuity of displacements  $u(\mathbf{r})$ and (or) \newline  
ii) Continuity of surface tractions ${P}$
\begin{eqnarray}
&&{P}\left( \mathbf{r}\right) \left. =\right. \widehat{{n}}%
\cdot \widetilde{{\tau} }\left( \mathbf{r}\right) {\mathbf{-}\widehat{{n}%
}\otimes \widehat{{n}}\left. :\right. \frac{\partial \widetilde
{{\mu} }%
}{\partial n}\left( \mathbf{r}\right) -\widehat{{n}}\cdot \left(
\nabla _{S}\cdot \widetilde{{\mu} }\left( \mathbf{r}\right) \right)} 
\notag \\
&&
-\widehat{{n}}\cdot \left( \nabla _{S}\cdot \widetilde{{\mu} }^{213}\left( 
\mathbf{r}\right) \right)   
\notag \\
&&
+\left[ \left( \nabla _{S}\cdot \widehat{{n}}\right) \widehat{%
{n}}\otimes \widehat{{n}}{-}\left( \nabla _{S}\widehat{%
{n}}\right) \right] \left. :\right. 
\widetilde{{\mu} }\left( \mathbf{r}%
\right)  \notag \\
&&+\widehat{{n}}\cdot \rho ^{\prime }\widetilde{D}\left. :\right. 
\frac{\partial ^{2}}{\partial t^{2}}\left( \widehat{{n}}\frac{%
\partial }{\partial n}u\left( \mathbf{r}\right) +\nabla _{S}%
u\left( \mathbf{r}\right) \right) \left. \right. ,\,\, \left. \right. \label{tractionjune}
\end{eqnarray}
where $\rho ^{\prime }$ stands for the mass of micro-material per unit macro-volume, $\widetilde{D}$ is a specific tensor
of fourth order depending on the physical and geometrical parameters of the microstructure and $\nabla _{S}=\nabla-\widehat{{n}} \frac{\partial}{\partial n}$ is the surface gradient. 
 
b) The non classical conditions: \newline
i) Continuity of surface double stresses ${\widetilde{R}\left( \mathbf{r}\right) \left. = \right. \widehat{{n}}
\cdot \widetilde{{\mu} } \left( \mathbf{r}\right) \cdot \widehat{{
n}}}$ and (or) \newline  
ii) Continuity of normal derivatives of the displacements $\frac{\partial u}{\partial n} \left( \mathbf{r} \right).$  
\newline
It is clear that even the classical boundary conditions are influenced by the presence of the microstructure, in the sense that the surface traction is "polluted" drastically by several extra terms (see Eq.(\ref{tractionjune})), owing their existence to the double stress tensor. Thus, applying macroscopically a specific mechanical load on the surface of a medium with microstructure, this load is absorbed and compensated by the structure after being suitably distributed to a classical Cauchy type surface traction along with additional traction terms incorporating the gradient type dynamic response of the material.

\section{The pseudo - differential method} 
\label{pseudo}
The framework of the pseudo-differential calculus is powerful but extended, complicated and very tough to be exposed extensively. This would be very disorientating for the purposes of the present work. Thus the structure of this section is somehow reversed. We begin with the original problem under investigation, we apply general results from the theory of pseudo-differential operators, wherever necessary, but focus mainly on the gradient elasticity problem itself. Just before the theorem stated at the end of the section, we present some necessary theoretical issues from this calculus that are indispensable for the acquisition of the stated result and could also clarify some delicate points of the first portion of the section.  

Let us then state that the development of the methodology  has been implemented for the static case and in particular we will be confined herein to the simplest gradient model with ${\xi}_{1}={\xi}_{2}=g$. Thus, the differential equation becomes
\begin{eqnarray}
(1-{{g}}^2 \Delta)\ \Delta^{\ast} \cdot u =-f \ \ \ \text{in}\ \  D \label{basicequation}.
\end{eqnarray}
This model, although simple, is very effective and indicative for the advantages and the adequacy of the gradient theory to applications (\cite{Aifantis:2005},\cite{Aifantis:1992},\cite{Aifantis:1993}). In fact, we encounter in \cite{Aifantis:1993} an efficient approach aiming at decoupling Eq.(\ref{basicequation}) in two consecutive second order equations in a natural manner as predicted by elementary differential equation calculus. However, in that approach, the differential equations are examined separately from the accompanying boundary conditions. These conditions are imposed at the final stage of the method and the point is that they do not belong to the set of conditions produced by the calculus of variations and then they do not contribute to minimization of the stored energy. In this section we will show that if someone decides to comply with the allowable set of boundary conditions provided in last section, then the standard effort to replace Eq.(\ref{basicequation}) by a two stage scheme of two second order equations leads necessarily to the emergence of intermediate boundary conditions of non-local type.   

Indeed let us consider one of the possible cases concerning the situation on the boundary $S=\partial D$ as follows: 
\begin{eqnarray}
\gamma_{0}\ u&=& 0 \ \ \ \text{on} \ \ \partial D \label{bc_correct_1} \\
\gamma_{{2}}\ u&:=& \hat{{n}}\ \hat{{n}} \ : {{\mathcal{H}}} \ \vdots \ {{\nabla}}{{\nabla}} u  = 0 \ \ \ \text{on} \ \ \partial D .\label{bc_correct_2} 
\end{eqnarray}
The symbol $\gamma$ denotes the trace operator acting on functions defined on $D$ and taking values on the trace of these functions on $\partial D$, while the subscript indicates its order \cite{lean:2000}. So, for example, for smooth functions on $\overline{D}$, $\gamma_{0}u=u{|}_{\partial D}$. Eq.(\ref{bc_correct_2}) is identical with the condition b(i) of Section \ref{sec:pereliminaries} implying $\widetilde{R} \left( \mathbf{r} \right)=0 \ $ on $\ \partial D$, as easily is deduced after using Eq.(\ref{represh}). The system (\ref{bc_correct_1},\ref{bc_correct_2}) is a homogeneous set of conditions but as it is well known no loss of generality occurs since any non homogeneous boundary stimulus could be "transfered" to the volume excitation $f$ by an appropriate transformation justified by the linearity of the problem. So we consider vanishing of displacements and double stresses on the surface of the region under investigation. Following the terminology concerning higher order boundary value problems (see for instance \cite{wloka:1987}) the boundary value operators induced by Eqs.(\ref{bc_correct_1},\ref{bc_correct_2}) constitute a normal system since the orders of the participants are different and less that $2m-1$ where $2m (=4)$ is the order of the differential equation. \newline
Gathering differential and boundary operators we form the following boundary value problem
\begin{eqnarray}
&& {{g}}^2 \Delta\ \Delta^{\ast} \cdot u_{{g}}-\Delta^{\ast} \cdot u_{{g}} = f \ \ \ \text{in}\ \  D  \nonumber \\
&& \gamma_{0}\ u_{{g}}= 0 \ \ \ \text{on} \ \ \partial D \nonumber \\
&& \gamma_{2}\ u_{{g}}= \hat{{n}}\ \hat{{n}} \ : {{\mathcal{H}}} \ \vdots \ {{\nabla} {\nabla} u_{{g}} } = 0 \ \ \ \text{on} \ \ \partial D, \label{perturbed}
\end{eqnarray}
where the displacement field has been assigned a subscript indicating the dependence on the gradient parameter $g$. The problem (\ref{perturbed}) will be nominated (perturbed) Problem I and can be considered as a perturbation of the classical elasticity (unperturbed) Problem II :
\begin{eqnarray} 
&& -\Delta^{\ast} \cdot u = f \ \ \ \text{in}\ \  D \nonumber \\
&& \gamma_{0}\ u = 0 \ \ \ \text{on} \ \ \partial D. \label{unperturbed}
\end{eqnarray}   

Furthermore, we consider the well posed auxiliary second order problem 
\begin{eqnarray}
- \Delta^{\ast} \cdot w &=& \chi \ \ \ \text{in}\ \  D \nonumber \\
 \gamma_{0}\ w&=& \phi \ \ \ \text{on} \ \ \partial D \label{classicproblem}
\end{eqnarray} 
It disposes the inverse operator
\begin{eqnarray*}
{\left(
\begin{array}{clrr}      
  - \Delta^{\ast} \\ \gamma_{0}    
\end{array} \right)}^{-1} =  \left( \begin{array}{clrr}      
  R_{0} & K_{0} \end{array} \right)
\end{eqnarray*} 
where $R_{0}$ (resp. $K_{0}$) is the operator solving the above problem for $\phi=0$ (resp. $\chi=0$). Usually $R_{0}$ is called the Green's operator (defined on $D$) while $K_{0}$ stands for Poisson operator (going from $\partial D$ to $D$). Speaking generally, if we deal with $n$ independent variables ($n>2$), then $R_{0}=R+R_{1}$, where $(R\chi) (x)= c_n {\int}_{D}  {\frac{\chi(y)}{{|x-y|}^{n-2}} dy}, \ \ x \in D \ \ $ and $R_{1}$ has singularities outside the closure of the domain $D$ \cite{lean:2000}. The Green's operator $R_{0}$ disposes a singular kernel and this singularity can be "measured" in the context of integral operators. But if we desire to face also the Poisson operator which has a similar structure but acts on a manifold or handle compositions of integral with differential operators it is necessary to embed these operators in the general framework of pseudo-differential operators. In fact following the basic terminology of \cite{Grubb:1986}, the Green's operator defines a $ps.d.o.$ (pseudo-differential operator) of order $-2$. Roughly speaking, the same is valid for the Poisson operator, although a rigorous analysis would evoke the theory of pseudo-differential operators defined on manifolds, which is out the scope of the present work.   \newline
All this stuff can be exploited in the following way: We refer to our fundamental Problem I and consider that due to the homogeneous boundary condition $\gamma_{0}\ u_{{g}} = 0$, we can write $u_{{g}}=R_{0}w$ for some $w$ (recall that $-\Delta^{\ast} \cdot R_{0}w = w$). We insert this in the first equation of (\ref{perturbed}) and put ${{s}}^{2}={{g}}^{-2}$. Then the Problem I transforms to the equivalent Problem III: 
\begin{eqnarray}
-\Delta w +{{s}}^2 w = {f}_{{{s}}} \ \ \ \text{in}\ \  D  \nonumber \\
\gamma_{2}\ R_{0} w = 0 \ \ \ \text{on} \ \ \partial D  \label{IIIEQ}
\end{eqnarray}
with $f_{{{s}}}={{s}}^2 f$. Remark that the boundary condition of the new problem is expressed via the operator  
\begin{eqnarray}
\gamma_{2}\ R_{0} = \hat{{n}}\ \hat{{n}} \ : {{\mathcal{H}}} \ \vdots \ {\nabla} {\nabla} R_{0},
\end{eqnarray}
which is a specific [\cite{{Grubb:1986}}] pseudo-differential operator ($ps.o.p.$) constituting the composition of a differential operator $\left( \hat{{n}}\ \hat{{n}} \ : {{\mathcal{H}}} \ \vdots \ {\nabla} {\nabla} \right)$ of order $2$ with an integral operator $R_{0}$ of order $-2$. Hence this product is expected to have zero order given the general rule of pseudo-differential calculus stating that under some broad hypotheses the order of the composition of two operators is the sum of the orders. Indeed following manipulations of pseudo-differential calculus \cite{{Grubb:1986}} by working with the symbol of these operators - see the general discussion below - we decompose this boundary operator as the following superposition
\begin{eqnarray*}
\gamma_{2}\ R_{0} =\gamma_{0}+T_{0}
\end{eqnarray*}  
where $T_{0}$ is a $ps.o.p$ of order $0$ and the first term of the composition is just the simple trace operator $\gamma_{0}$. Thus the Problem III could be slightly reformulated as follows 
\begin{eqnarray}
-\Delta w +{{s}}^2 w = {f}_{{{s}}} \ \ \ \text{in}\ \  D  \nonumber \\
\gamma_{0} w + T_{0} w = 0\ \ \ \text{on} \ \ \partial D  \label{IIIEQ1}
\end{eqnarray}
and this is exactly the second order alternative of Problem I. The price to pay is the appearance of the term $T_{0} w$ in the boundary condition. The operator $T_{0}$ has zero order and so disposes easy handling but it represents a not local operator depending on the geometrical characteristics of the domain $D$.  \newline
Having substituted the original fourth-order differential Problem I by the second order pseudo-differential Problem III, we are in position to investigate the stability of the original problem as $g \rightarrow 0$ or equivalently as $s \rightarrow \infty$ and the classical problem II is obtained. The establishment of convergence and the estimation of the rate of the convergence passes through the study of the symbol of the involved operators. 

There is no space for the entire analysis to be presented here but briefly we are obliged to support our argumentation by giving the necessary scientific framework, which might also clarify some statements above. If we are referring to free space $R^{n}$, every $ps.o.p$ $P$ has the form 
\begin{eqnarray}
Pu(x)&=&{(2 \pi)}^{-n}{\int}_{R^{2n}}{e^{i(x-y) \cdot \xi}p(x,\xi) u(y) dy d \xi}, \notag \\
&& x \in R^{n} 
\end{eqnarray} 
where $p(x,\xi)$ is the symbol of the operator and we denote $P=OP(p(x,\xi))$. As an example the inverse operator ${(1-\Delta)}^{-1}$ of the modified Helmholtz operator $1-\Delta$ has the symbol function $p(x,\xi)=p(\xi)=\frac{1}{(1+{|\xi|}^{2})}$ and the degree of the denominator imposes the order $d=-2$ for this operator. We give this example to notify that for powers of linear operators with constant coefficients in free space, the pseudo-differential calculus is equivalent to Fourier transform analysis. The situation changes drastically when boundaries are present. The calculus of Boutet de Monvel \cite{Monvel:1971} is a solution to the problem of establishing a class of operators encompassing the elliptic boundary value problems as well as their solution operators (the inverses of the differential operators and the conditions). Moreover the suggested class of operators is closed under composition. (It is an "algebra"). To give a brief introduction to this stuff let us consider our region $D$ with its smooth boundary $\partial D$. We consider a local chart on the surface $\partial D$ incorporating a standard partition of unity, which is a straightforward manner to represent the manifold and to transform locally the closure $\overline{D}$ via local diffeomorphisms on the half space ${\overline{R}}^{n}_{+}=R^{n-1} \times \overline{R}_{+}$. Thus establishing the calculus of Boutet de Monvel just in ${\overline{R}}^{n}_{+}$ is not restrictive since the inverse diffeomorphism aids at transferring everything back on $\overline{D}$. The standard construction of a $ps.d.o$ on $\Omega=R^{n}_{+}$ is expected to be the "restriction" of a $ps.d.o$ $P$ on $R^{n}$ as follows: $P_{\Omega}u:= r^{+} P e^{+} u$ 
where $r^{+}$ ($e^{+}$) is the restriction (extension) by zero operator. Unfortunately this definition is not enough since generally the discontinuity of $e^{+}u$ on $x_{n}=0$ causes singularities and the operator $P_{\Omega}$ fails to map $H_{comp}^{m}(\overline{\Omega})$ into $H_{comp}^{m-d}(\overline{\Omega})$. Boutet de Monvel singled out a class of $ps.d.o.s$ where the mapping properties of $P_{\Omega}$ are nice, namely the $ps.d.o.s$ having the transmission property. This property means that if $p(x^{\prime},x_{n},{\xi}^{\prime},\xi_{n})$ is the symbol of $P$ then the one dimensional inverse Fourier transforms ${\tilde{p}}_{\alpha, \beta}(x^{\prime},x_{n},{\xi}^{\prime},z_{n})= {\mathcal{F}}^{-1}_{\xi_{n} \rightarrow z_{n}} \left( D_{x}^{\beta}D_{\xi}^{\alpha} p(x^{\prime},x_{n},{\xi}^{\prime},\xi_{n}) \right)$, 
after being restricted on  $x_{n}=0$, remain $C^{\infty}$ functions both as $z_{n} \rightarrow 0+$ and $z_{n} \rightarrow 0-$ (for all $\alpha, \ \beta \in N^{n}$).  

Keeping on working with the semi-infinite space as justified above, we are now in position to formulate the necessary arguments interrelating the perturbed and unperturbed problems in the regime of stability examination and convergence estimation. We refer to the perturbed problem I and consider the mapping produced by Fourier transforming all differentiations except the normal derivative. We obtain the transformation
\begin{eqnarray}
 \left( \begin{array}{cccc}       
  (g^2 (D_{n}^{2}+ {|{\xi}^{\prime}|}^{2})+1)\left( \mu (D_{n}^{2}+{|{\xi}^{\prime}|}^{2}) \mathcal{I}+(\lambda +\mu)\mathcal{A} \right) \\ 
  \gamma_{0} \\ \gamma_{2}
\end{array} \right) \nonumber \\
: {(H^{4}({\overline{R}}_{+}))}^{n} \rightarrow 
\left( \begin{array}{clrr}      
   {(L_{2}({R}_{+}))}^{n} \\ \times \\ C \\ \times \\ C   
\end{array} \right) \nonumber  
\end{eqnarray}
with $\mathcal{A}={({\xi}^{\prime},0)}^{T}({\xi}^{\prime},0)+{({\underline{0}}_{n},D_{n})}^{T}({\underline{0}}_{n},D_{n}), \ \ \mathcal{I}={\hat{x}}_{j}{\hat{x}}_{j}$. \newline
This transformation induces in a standard manner \cite{wloka:1987} the corresponding sesquilinear form which is easily proved to be bounded below. Consequently the mapping above is bijective for $\xi^{\prime} \neq 0$ and $g>0$ (and of course $s > 0$). Furthermore, composing with the boundary symbol $r_{0}(x,{\xi}^{\prime},D_{n})$ of $R_{0}$, we can show that the boundary symbol operator of the Problem III satisfies the main assumptions I and II of Definition 1.5.5 of \cite{Grubb:1986}, which definition states the necessary conditions for the validity of the parameter-ellipticity . In addition the assumption III of the same definition is valid since the limit for  ${\xi}^{\prime} \rightarrow 0$ of the strictly homogeneous boundary symbol operator for (\ref{IIIEQ}) is 
\begin{eqnarray*}
\left( \begin{array}{cccc}       
  D^{2}_{n}+s^2 \\ 
  \gamma_{0} 
\end{array} \right),
\end{eqnarray*}
which is invertible for $s>0$. Apart the parameter ellipticity is established, the theorems of Section 3.3 of \cite{Grubb:1986} are applied to (\ref{IIIEQ}). Thus we obtain unique solvability for $s \geq s_{0}$ (for some $s_{0} > 0$) and $s$-dependent estimates of the solution $w$ in terms of $f_{s}$, that can be used to give $g$-estimates of $u_{g}$ or $u_{g}-u$ in terms of the excitation $f$, with precise information about the $g$-dependence.  

The implementation of this convergence rate concerns the general inversion problem (with possibly not homogeneous boundary conditions)  
\begin{eqnarray*}
{\left( \begin{array}{cccc}       
  {{g}}^2 \Delta\ \Delta^{\ast} \cdot u_{{g}}-\Delta^{\ast} \cdot u_{{g}} \\ 
  \gamma_{0} \\ \gamma_{2}
\end{array} \right)}^{-1}=(M_{g},N_{g,0},N_{g,2})
\end{eqnarray*}
where $M_{g}=R_{0}+{M_{g}^{\prime}}$ ($N_{g,0}=K_{0}+{N_{g,0}^{\prime}}$) is a perturbation of Green's operator $R_{0}$ (Poisson operator $K_{0}$). The quantification of the arguments above is accomplished in a tedious but straightforward manner by applying, as stated before, the basic theorem of Section 3 of \cite{Grubb:1986} to the operators of gradient elasticity exposed above. We give so the final outcome of this analysis quantifying the convergence analysis 
\newtheorem{env_name}{Theorem}
\begin{env_name}
The problem 
\begin{eqnarray*}
&&{{g}}^2 \Delta\ \Delta^{\ast} \cdot u_{{g}}-\Delta^{\ast} \cdot u_{{g}} = f \ \ \ \text{in}\ \  D  \\
&& \gamma_{0}\ u_{{g}}= \phi_{0} \ \ \ \text{on} \ \ \partial D \\
&& \gamma_{2}\ u_{{g}}= \hat{{n}}\ \hat{{n}} \ : {{\mathcal{H}}} \ \vdots \ {{\nabla} {\nabla} u_{{g}} } = \phi_{2} \ \ \ \text{on} \ \ \partial D 
\end{eqnarray*}
has a unique solution in closeness with the solution $u=R_{0}f+K_{0}\phi_{0}$ of the unperturbed problem in the sense that $u_{{g}}=u+{M_{g}^{\prime}} f+{N_{g,0}^{\prime}} \phi_{0} + {N_{g,2}^{\prime}} \phi_{2}$ with 
\begin{eqnarray*}
{\|{M_{g}^{\prime}} f\|}_{t} \leq c  {{g}}^{3-t+l}{\| f\|}_{1+l}+c_{\delta} {{g}}^{\frac{5}{2}-t}{\| f\|}_{\frac{1}{2}+\delta}, \\  1 \leq t \leq 3+l, \ \ l \geq 0,\ \ \delta>0. \\
{\|{N_{g,0}^{\prime}} \phi_{0} \|}_{t} \leq c  \left( {{g}}^{3-t+l}{\| \phi_{0}\|}_{\frac{3}{2}+l}+{{g}}^{\frac{5}{2}-t}{\| \phi_{0}\|}_{2}\right), \\ 1 \leq t \leq 3+l, \ \ l \geq 0 \\
{\|{N_{g,2}^{\prime}} \phi_{2} \|}_{t} \leq c  \left( {{g}}^{\frac{5}{2}-t}{\| \phi_{2}\|}_{0}+{{g}}^{\frac{5}{2}-t+l}{\| \phi_{2}\|}_{l}\right), \\ 1 \leq t \leq \frac{5}{2}+l, \ \ l \geq 0
\end{eqnarray*}
\end{env_name}
Every norm of the type ${\| \cdot \|}_{t}$ appeared in the Theorem refers to the Sobolev space $H^{t}$ on the appropriate domain. As a clarifying example, if we are interesting in estimating the regularity of the response of the homogeneous problem ($\phi_{0}=\phi_{2}=0$) to the force $f \in H^{2}(D)$, the first estimate of the theorem provides the bound ${\|u_{g}-u \|}_{2} \leq C g^{\frac{1}{2}} {\|f \|}_{2}$.

\section{The purely differential approach} 
\label{alter}
In this section we present a different approach leading to a second order purely differential scheme referring though to an increased number of independent variables. More precisely, we consider as unknowns of the problem not only the displacement field $u$ but also the gradient of the strain tensor $\widetilde{\nu}:= \nabla \widetilde{e}$, which is a triadic symmetric in the two last indices. The variables $u,\widetilde{\nu}$ are considered initially independent and their connection will be built from the beginning. First we reformulate the primitive Eq.(\ref{basceq}), which on the basis of the constitutive law (\ref{represh}) becomes    
\begin{eqnarray}
\Delta^{\ast} \cdot u -\nabla \nabla\ : \mathcal{H} \ \vdots \ \widetilde{\nu} &=& -f. \label{basceq1}
\end{eqnarray}
Moreover the variables $u, \widetilde{\nu}$ should be forced to satisfy 
\begin{eqnarray}
\frac{1}{2} \nabla \left( \nabla u + {(\nabla u)}^{T} \right) -\widetilde{\nu} =0. \label{basceq2}
\end{eqnarray}
We would like to assemble the last two equations to form a differential system of second order sharing useful properties of the corresponding calculus.

First, we introduce the differential tensor of second order and fourth degree $\mathcal{D}:=\frac{1}{2}\left( I \nabla \nabla + {(I \nabla \nabla)}^{1324} \right)=\frac{1}{2}\left( I \nabla \nabla + ({\hat{x}}_{i} \nabla {\hat{x}}_{i} \nabla) \right)$ and remark that 
\begin{eqnarray*}
{\mathcal{D}}^{T} \cdot u:={\mathcal{D}}^{4321} \cdot u = \nabla \widetilde{e} = \widetilde{\nu},
\end{eqnarray*}
which is another way to state Eq.(\ref{basceq2}). 
We transform then Eqs.(\ref{basceq1},\ref{basceq2}) into the following two alternative systems 
\begin{eqnarray}
 \left( \begin{array}{cccc}
 - \Delta^{\ast} \ \cdot &  \mathcal{D} \ \vdots \ \mathcal{H} \ \vdots \\
 
  \pm \ {\mathcal{H}} \ \vdots \ {\mathcal{D}}^{T} \cdot & \mp \ {\mathcal{H}} \ \vdots
  \end{array} \right) 
   \left( \begin{array}{cccc}
   u \\ \widetilde{\nu}
   \end{array} \right)=\left( \begin{array}{cccc}
   f \\ 0
   \end{array} \right) \label{forms}
\end{eqnarray} 
The operators 
\begin{eqnarray}
 A_{\pm} := \left( \begin{array}{cccc}
 - \Delta^{\ast} \ \cdot &  \mathcal{D} \ \vdots \ \mathcal{H} \ \vdots \\ 
  \pm \ {\mathcal{H}} \ \vdots \ {\mathcal{D}}^{T} \cdot & \mp \ {\mathcal{H}} \ \vdots
  \end{array} \right) \nonumber \\ 
  : D(A_{\pm}) \subset {(H^{1}(D))}^{3} \times {(H_{sym}^{1}(D))}^{27} \nonumber \\ 
  \rightarrow {(L^{2}(D))}^{3} \times {(L_{sym}^{2}(D))}^{27}
\end{eqnarray} 
are induced explicitly from the formulations (\ref{forms}), dispose domains $D(A_{\pm})$ (to be specified later on) that are subsets of the appeared Sobolev spaces. The subscript $sym$ denotes that from all possible triadics belonging to ${(H^{1}(D))}^{27}$ (or ${(L^{2}(D))}^{27}$), we refer to those that are symmetric in the last two indices.   

Both systems (\ref{forms}) represent static gradient elasticity via a polyadic differential operator of order two. The physical meaning of this reduction is that the constitutive equations of the problem have been absorbed by the differential system itself, have become a part of the PDE's governing the motion of the region with microstructure and thus the systems (\ref{forms}) do not need any additional external connection between the involved variables. What remains to be implemented is the set of boundary conditions of the problem, which of course must comply again with the calculus of variations as presented in Section \ref{sec:pereliminaries}. The boundary conditions will be responsible for the determination of the domains $D(A_{\pm})$ as usually happens in the functional theoretic setting of the BVP's. 

Before proceeding to the characterizations of the domains of the operators under discussion, we would like to distinguish the two operators $A_{\pm}$ in the following heuristic manner: Due to the symmetry of the constitutive polyadic $\mathcal{H}$ the operator $A_{+}$ is formally symmetric while $A_{-}$ can be decomposed as the sum of a symmetric and an antisymmetric component. So just at first sight the operator $A_{+}$ might be considered as more convenient than $A_{-}$ since the symmetric operators have generally useful properties. However this is not the case as will be proved in the sequel and if someone wanted to insist in symmetric realizations, then a drastic change of $A_{+}$ would be necessary.  

To build the functional theoretic establishment of the current section, we aim at connecting the operators $A_{\pm}$ with the sesquilinear form induced by the calculus of variations. We start with two pairs $(u_{i},{\widetilde{\nu}}_{i}),\ \ i=1,2$ and form the usual $L^{2}-$ inner product 
\begin{eqnarray}    
&&\left\langle (u_{1},{\widetilde{\nu}}_{1}),\ A_{\pm} \ (\begin{array}{cccc} u_{2} \\ {\widetilde{\nu}}_{2} \end{array}) \right\rangle := -\int_{D} \overline{u}_{1} \cdot \Delta^{\ast} \cdot {u}_{2} +  \nonumber \\
&&\int_{D} \overline{u}_{1} \cdot  \mathcal{D} \ \vdots \ \mathcal{H} \ \vdots \  {\widetilde{\nu}}_{2} \pm \int_{D} {\overline{\widetilde{\nu}}}_{1}^{321} \ \vdots \ \mathcal{H} \ \vdots \ {\mathcal{D}}^{T} \cdot u_{2} \nonumber \\
&&\mp \int_{D} {\overline{\widetilde{\nu}}}_{1}^{321} \ \vdots \ \mathcal{H} \ \vdots \ {\widetilde{\nu}}_{2} \label{inner}
\end{eqnarray}
One integration by parts to any one except the final integral participating in Eq.(\ref{inner}), the introduction of the auxiliary fields ${\widetilde{\mu}}_{i}=\mathcal{H} \ \vdots \ {\widetilde{\nu}}_{i},\ i=1,2$ and some straightforward manipulations transform this equation to the next form  
\begin{eqnarray}    
&&\left\langle (u_{1},{\widetilde{\nu}}_{1}),\ A_{\pm} \ (\begin{array}{cccc} u_{2} \\ {\widetilde{\nu}}_{2} \end{array}) \right\rangle =W_{el}(u_{1},u_{2}) \nonumber \\
&&-\int_{D} (\nabla \cdot {\mu}_{2})\ : \ \nabla {\overline{u}}_{1} \mp \int_{D} (\nabla \cdot {\overline{\widetilde{\mu}}}_{1})\ : \ \nabla {{u}}_{2}  \nonumber \\ 
&&- \int_{S} {\overline{u}}_{1} \cdot {{P}}_{2} - \int_{S} {\overline{u}}_{1} \cdot \left\{ [\nabla_{S}\hat{n}-(\nabla_{S} \cdot \hat{n})\hat{n}\hat{n}]: {{\widetilde{\mu}}}_{2} \right. \nonumber  \\
&& \left. +\hat{n}{\nabla}_{S}: {{\widetilde{\mu}}}_{2}^{213} \right\}  \mp \int_{S} {{u}}_{2} \cdot \left\{ [\nabla_{S}\hat{n}-(\nabla_{S} \cdot \hat{n})\hat{n}\hat{n}]: {\overline{\widetilde{\mu}}}_{1} \right. \nonumber  \\
&& \left. +\hat{n}{\nabla}_{S}: {\overline{\widetilde{\mu}}}_{1}^{213} \right\} \pm \int_{S} \frac{\partial u_{2}}{\partial n} \cdot (\hat{n}\hat{n} : {\overline{\widetilde{\mu}}}_{1}) \nonumber \\
&& \mp \int_{D} {\overline{\widetilde{\nu}}}_{1}^{321} \ \vdots \ \mathcal{H} \ \vdots \ {\widetilde{\nu}}_{2},  \label{sequil}
\end{eqnarray}
where ${{P}}_{i},\ i=1,2$ is the static surface traction -see relation (\ref{tractionjune})-pertaining to the pair $(u_{i}, {\widetilde{\nu}}_{i})$ and $W_{el}$ is the well known sesquilinear form of classical elasticity \cite{lean:2000}
\begin{eqnarray}
W_{el}(u_{1},u_{2})=\int_{D} [2 \mu e_{jk}({\overline{u}}_{1})e_{jk}({{u}}_{2})+\lambda \text{div}{\overline{u}}_{1} \text{div}{{u}}_{2}]
\end{eqnarray}
Performing one more integration by parts to the integrals appearing in the right hand side of Eq.(\ref{sequil}), we obtain 
\begin{eqnarray}
&&\left\langle (u_{1},{\widetilde{\nu}}_{1}),\ A_{\pm} \ (\begin{array}{cccc} u_{2} \\ {\widetilde{\nu}}_{2} \end{array}) \right\rangle=\left\langle \ A_{\pm}^{\ast} \ (\begin{array}{cccc} u_{1} \\ {\widetilde{\nu}}_{1} \end{array}), (u_{2},{\widetilde{\nu}}_{2}) \right\rangle \nonumber \\
&&- \int_{S} {\overline{u}}_{1} \cdot {{P}}_{2} \pm \int_{S} {{u}}_{2} \cdot {\overline{{P}}}_{1} - \int_{S} \frac{\partial {\overline{u}}_{1}}{\partial n} \cdot (\hat{n}\hat{n} : {{\widetilde{\mu}}}_{2}) \nonumber \\
&&\pm \int_{S} \frac{\partial u_{2}}{\partial n} \cdot (\hat{n}\hat{n} : {\overline{\widetilde{\mu}}}_{1})-(1 \pm 1) \int_{D} (T^{\ast} \cdot {\overline{u}}_{1}) \cdot u_{2} \nonumber \\
&&\mp \int_{D} {\overline{\widetilde{\nu}}}_{1}^{321} \ \vdots \ \mathcal{H} \ \vdots \ {\widetilde{\nu}}_{2}, \label{adjointjune}
\end{eqnarray}
where we encounter the Cauchy stress tensor $T^{\ast}$ and the formal adjoint operator $A_{\pm}^{\ast}$. All the surface integrals appearing in the relation above pertain to traces of the involved fields on the surface $S=\partial D$. We remark that two possible sets of boundary conditions arise that annihilate these surface integrals. The first choice is 
\begin{eqnarray}
&&D({A}_{\pm}^{(1)})=D({({A}_{\pm}^{(1)})}^{\ast})=\left\{ (u,\widetilde{\nu}) \in \mathcal{L}:={(H^{1}(D))}^{3} \times \right. \nonumber \\ 
&&\left. {(H_{sym}^{1}(D))}^{27}:  u{|}_{S}=0 \ \wedge \ \hat{n}\hat{n}: \mathcal{H}\ \vdots \ {\widetilde{\nu}} {|}_{S}=0 \right\}   \label{mpros}
\end{eqnarray}
Clearly ${A}_{\pm}^{(1)}$ is densely defined in $\mathcal{L}$. Thus, in particular ${A}_{+}^{(1)}$ is a self-adjoint operator while ${A}_{-}^{(1)}$ is not due to the fact that ${A}_{-}^{(1)}$ and ${({A}_{-}^{(1)})}^{\ast}$ are different differential operators. Let us recall here that the boundary conditions defining the structure of ${A}_{\pm}^{(1)}$ are exactly the conditions (a1-b1) of Section \ref{sec:pereliminaries}.  To discuss the second choice of boundary conditions, let us remark that when $u{|}_{S} \neq 0$ then the last term of Eq.(\ref{adjointjune}) can not be compensated in the case of $A_{+}$ unless the unphysical - in the realm of gradient elasticity - condition $T^{\ast} \cdot {\overline{u}}_{1}=0$ holds. On the contrary $A_{-}$ is well behaving and admits the following alternative set of conditions
\begin{eqnarray}
&&D({A}_{-}^{(2)})=D({({A}_{-}^{(2)})}^{\ast})=\left\{ (u,\widetilde{\nu}) \in \mathcal{L}:={(H^{1}(D))}^{3} \times \right. \nonumber \\ 
&&\left. {(H_{sym}^{1}(D))}^{27}:  {P}{|}_{S}=0 \ \wedge \ \frac{\partial u}{\partial n} {|}_{S}=0 \right\}   \label{pisw}
\end{eqnarray}
Then the boundary conditions (a2-b2) of Section \ref{sec:pereliminaries} have been emerged naturally via the second order formulation and define the operator $A_{-}^{(2)}$, which although is densely defined and satisfies $D({A}_{-}^{(2)})=D({({A}_{-}^{(2)})}^{\ast})$, is not self adjoint as explained before. 
The regularity of the solutions imposed a priori by Eqs.(\ref{mpros},\ref{pisw}) is the appropriate once since it stems as usually from the sesquilinear form which is interrelated with the second order differential operator. In our case we refer to Eq.(\ref{sequil}) - examining from now on only the (-) case - and after imposing boundary conditions (of any type) we nominate the right hand side as the sesquiliear form
\begin{eqnarray}
&&\Phi\left((u_{1},{\widetilde{\nu}}_{1}),(u_{2},{\widetilde{\nu}}_{2})\right):= W_{el}(u_{1},u_{2}) + \int_{D} {\overline{\widetilde{\nu}}}_{1}^{321} \ \vdots \ \mathcal{H} \ \vdots \ {\widetilde{\nu}}_{2}\nonumber \\
&&-\int_{D} (\nabla \cdot {\mu}_{2})\ : \ \nabla {\overline{u}}_{1} + \int_{D} (\nabla \cdot {\overline{\widetilde{\mu}}}_{1})\ : \ \nabla {{u}}_{2}  \nonumber \\ 
&& - \int_{S} {\overline{u}}_{1} \cdot \left\{ [\nabla_{S}\hat{n}-(\nabla_{S} \cdot \hat{n})\hat{n}\hat{n}]: {{\widetilde{\mu}}}_{2} +\hat{n}{\nabla}_{S}: {{\widetilde{\mu}}}_{2}^{213} \right\} \nonumber  \\
&& + \int_{S} {{u}}_{2} \cdot \left\{ [\nabla_{S}\hat{n}-(\nabla_{S} \cdot \hat{n})\hat{n}\hat{n}]: {\overline{\widetilde{\mu}}}_{1} +\hat{n}{\nabla}_{S}: {\overline{\widetilde{\mu}}}_{1}^{213} \right\}  \nonumber
\end{eqnarray}    
and so the operators $A_{-}^{(l)},\ l=1,2$ are rigorously defined via the relation
\begin{eqnarray}    
&&\left\langle (u_{1},{\widetilde{\nu}}_{1}),\ A_{-}^{(l)} \ (\begin{array}{cccc} u_{2} \\ {\widetilde{\nu}}_{2} \end{array}) \right\rangle =\Phi\left((u_{1},{\widetilde{\nu}}_{1}),(u_{2},{\widetilde{\nu}}_{2})\right), \nonumber \\
&&\forall (u_{j},{\widetilde{\nu}}_{j}) \in D(A_{-}^{(l)}), \ \ j,l=1,2
\end{eqnarray}
The sesquilinear form $\Phi$ merits some very interesting properties. First, it is bounded in the sense that 
\begin{eqnarray}
|\Phi\left((u,{\widetilde{\nu}}),(v,{\widetilde{w}})\right)| \leq C \| (u,{\widetilde{\nu}}){\|}_{{}_{1}}\| (v,{\widetilde{w}}){\|}_{{}_{1}},
\end{eqnarray}
where $\| (u,{\widetilde{\nu}}){\|}_{{}_{1}}^{2}=\|u{\|}_{H^{1}(D)}^{2}+\|\widetilde{\nu}{\|}_{H^{1}(D)}^{2}$.
The second property is crucial and induces some kind of lower boundedness for the sesquilinear form. Indeed on the basis of Eq.(\ref{ante}) and the coercivity of the classical elastic sesquilinear form \cite{lean:2000}, it holds that
\begin{eqnarray}
&&|\Phi\left((u,{\widetilde{\nu}}),(u,{\widetilde{\nu}})\right)| \geq \Re[\Phi\left((u,{\widetilde{\nu}}),(u,{\widetilde{\nu}})\right)] \nonumber \\
&&= W_{el}(u,u)+ \int_{D} {\overline{\widetilde{\nu}}}^{321} \ \vdots \ \mathcal{H} \ \vdots \ {\widetilde{\nu}} \nonumber \\
&&\geq c(\|u{\|}_{H^{1}(D)}^{2}+\|\widetilde{\nu}{\|}_{L^{2}(D)}^{2}) - C \|u{\|}_{L^{2}(D)}^{2} \label{ouf}
\end{eqnarray}  
for some positive constant $c, C$ dependent on the gradient parameters and the Lam$\acute{\text{e}}$'s constants only. In particular when we are dealing with the first kind of boundary conditions, i.e. $l=1$, involving the vanishing of $u$ on $S$, then the classical elasticity sesquilinear form $W_{el}$ is bounded below \cite{lean:2000} and then the positiveness of $\Phi$ emerges naturally:
\begin{eqnarray}
&&|\Phi\left((u,{\widetilde{\nu}}),(u,{\widetilde{\nu}})\right)| \geq \Re[\Phi\left((u,{\widetilde{\nu}}),(u,{\widetilde{\nu}})\right)] \nonumber \\
&&= W_{el}(u,u)+ \int_{D} {\overline{\widetilde{\nu}}}^{321} \ \vdots \ \mathcal{H} \ \vdots \ {\widetilde{\nu}} \nonumber \\
&&\geq c(\|u{\|}_{H^{1}(D)}^{2}+\|\widetilde{\nu}{\|}_{L^{2}(D)}^{2}), \ \  \forall (u,{\widetilde{\nu}}) \in D(A_{-}^{(1)}). \nonumber \\
&& \ 
\end{eqnarray}  
We are now in position to state the main outcome of the present section.
\newtheorem{sic}[env_name]{Theorem}
\begin{sic}
The system of second order differential equations  
\begin{eqnarray}
 \left( \begin{array}{cccc}
 - \Delta^{\ast} \ \cdot &  \mathcal{D} \ \vdots \ \mathcal{H} \ \vdots \\
 
  - \ {\mathcal{H}} \ \vdots \ {\mathcal{D}}^{T} \cdot &  \ {\mathcal{H}} \ \vdots
  \end{array} \right) 
   \left( \begin{array}{cccc}
   u \\ \widetilde{\nu}
   \end{array} \right)=\left( \begin{array}{cccc}
   f \\ 0,
   \end{array} \right),\ \text{in}\ D \label{forms2} \label{newnew}
\end{eqnarray} 
in conjunction with the boundary conditions 
\begin{eqnarray}
&&u{|}_{S}=0 \ \wedge \ \hat{n}\hat{n}: \mathcal{H}\ \vdots \ {\widetilde{\nu}} {|}_{S}=0  \label{newpas}
\end{eqnarray}
disposes exactly one solution $(u,\widetilde{\nu}) \in {(H^{1}(D))}^{3} \times {(H_{sym}^{1}(D))}^{27}$.  This solution is equal to $(u,\nabla \widetilde{e})=(u,\frac{1}{2}\nabla (\nabla u+{(\nabla u)}^{T}))$ where $u$
satisfies the original fourth order differential equation of gradient elasticity
\begin{eqnarray}
\Delta^{\ast} \ \cdot \ u - \nabla \nabla \ : \ \mathcal{H}\ \vdots \  \nabla \nabla \ u =-f,\ \ \text{in}\  D \label{orig}
\end{eqnarray}
with the boundary conditions
\begin{eqnarray}
&&u{|}_{S}=0 \ \wedge \ \hat{n}\hat{n}: \mathcal{H}\ \vdots \ \nabla \nabla u {|}_{S}=0.  
\end{eqnarray}
The consideration of the second set of boundary conditions ($ {P}{|}_{S}=0 \ \wedge \ \frac{\partial u}{\partial n} {|}_{S}=0 $) obeys to the rules of the Fredholm alternative.
\end{sic}

\begin{proof}
The differential equation (\ref{orig}) in accordance with the first of the permitted sets of boundary conditions ($l=1$), primarily thanks to coercivity (\ref{ante}), accepts \cite{wloka:1987} a unique solution $u$, which due to regularity of elliptic problems belongs actually to ${(H^{4}(D))}^{3}$. Then the tensor $\widetilde{w}=\frac{1}{2}\nabla(\nabla u+{(\nabla u)}^{T})$ belongs to ${(H_{sym}^{2}(D))}^{27} \subset {(H_{sym}^{1}(D))}^{27}$ and it is a  straightforward matter to prove that the pair $(u,\widetilde{w})$ satisfies the differential system (\ref{newnew}) with the accompanying boundary conditions. On the other hand this is the unique solution of the second order boundary value problem. Indeed if we supposed that there existed a second solution then the difference of the two solutions $(z,\widetilde{\chi})$ would satisfy the homogeneous version of the B.V.P (\ref{newnew}-\ref{newpas}). Then
\begin{eqnarray}
&&0 = \left\langle (z,{\widetilde{\chi}}),\ A_{-}^{(1)} \ (\begin{array}{cccc} z \\ {\widetilde{\chi}} \end{array}) \right\rangle =\Phi\left((z,{\widetilde{\chi}}),(z,{\widetilde{\chi}})\right), \nonumber \\
&& \geq c(\|z{\|}_{H^{1}(D)}^{2}+\|\widetilde{\chi}{\|}_{L^{2}(D)}^{2}) 
\end{eqnarray}
from where we deduce that $z=0,\ \widetilde{\chi}=\widetilde{0}$, which assures uniqueness. The same arguments can be repeated for non homogeneous boundary conditions via the simple concept of transforming the problem by incorporating the boundary conditions in the forcing term (on the basis of  Duhamel's principle).
For the second set of boundary conditions ($l=2$) the situation is usually more intrinsic. The homogeneous problem (with $f=0$) admits no trivial solutions. These solutions coincide with the rigid motions \cite{lean:2000}. These solutions can be excluded in the framework of deformable media theory and this can be accomplished in functional theoretical terms as suggested in \cite{acharalas:2002}. Then uniqueness is restored and the unique solvability of the inhomogeneous problem passes through the application of the Fredholm alternative, which is involved via the evocation of coerciveness (\ref{ouf}) \cite{wloka:1987}.  
\end{proof}
   
\section*{Conclusions}
Two alternative second order formulations for boundary value problems of gradient elasticity are presented in the current work. The first attempt leaded to the formulation of a two stage algorithm characterized mainly by the fact that one boundary condition has non local behavior. The implication of the pseudo-differential calculus is the appropriate methodology to establish stability and convergence analysis as gradient elasticity - considered as a perturbed state - gives gradually place to the classical elastic case. The second approach consists in the reformulation of the gradient elasticity problem via a second order differential system of several independent variables so that the constitutive equations of gradient elasticity have become part of the differential system itself. The settlement of appropriate operators sharing efficient coerciveness properties lead to verification of equivalence of the original fourth order problem with the suggested new system and gave the base for assuring uniqueness of the solution and solvability governed by the rules of the Fredholm alternative theory.      
\bibliographystyle{plain}
\normalem
\bibliography{gradient_new}
\end{document}